\documentclass[11pt,reqno,oneside]{amsart}
\usepackage[utf8]{inputenc}
\usepackage{geometry}
\usepackage{amsmath,tikz}
\usetikzlibrary{matrix}
\usepackage{graphicx}
\graphicspath{{pics/}{../pics/}}
\usepackage{amssymb}
\usepackage{url}
\usepackage{hyperref}
\usepackage{bm}

\newcommand{\pb}{\mathbb{P}} 
\newcommand{\reals}{\mathbb{R}} 
\newcommand{\cplx}{\mathbb{C}} 
\newcommand{\floor}[1]{\lfloor #1 \rfloor} 

\newtheorem{thm}{Theorem}
\newtheorem{pro}[thm]{Proposition}

\theoremstyle{remark}
\newtheorem{rem}{Remark}
\theoremstyle{definition}

\title[Polynomial approximation of Self-similar measures]{Polynomial approximation of  self-similar measures and the spectrum of the transfer operator}
\author{Christoph Bandt and Helena Pe\~na} 
\thanks{This work was supported by Deutsche Forschungsgemeinschaft grant Ba 1332/11-1.} 
\date{}
\begin{document}

\begin{abstract}
We consider self-similar measures on $\reals .$  The Hutchinson operator $H$ acts on measures and is the dual of the transfer operator $T$ which acts on continuous functions. We determine polynomial eigenfunctions of $T .$ As a consequence, we obtain eigenvalues of $H$ and natural polynomial approximations of the  self-similar measure. Bernoulli convolutions are studied as an example.
\end{abstract}

\maketitle

\section{Introduction}\label{intro}
We consider self-similar measures $\nu$ on $\reals,$ given by affine maps $f_i(x)=t_ix+v_i$ with $0<t_i<1,$ and probabilities $p_i$ for $i=1,...,m.$ A simple example are Bernoulli convolutions where we have two mappings on an interval, $t_1=t_2=t$ and $p_1=p_2=\frac12.$  The measure $\nu$ has compact support $X$ which is the self-similar set defined by the $f_i.$ It is the unique probability measure which is fixed  by the Hutchinson operator

\begin{equation}\label{hut}  
H\mu (B)=\sum_{i=1}^m p_i  \mu (f_i^{-1}(B))\qquad\mbox{ for Borel sets } B\subset X\, .\end{equation}

$H$ acts on the space of all finite Borel measures on $X.$ Hutchinson  \cite{hutchinson}  proved that for any initial probability measure $\mu_0$ the sequence $\mu_n=H\mu_{n-1}, n=1,2,...$ converges to $\nu$ geometrically, with factor $t=\max t_i ,$ with respect to a certain metric. Thus $1$ is an eigenvalue  of $H,$ with eigenvector $\nu ,$ and there should be no  
other eigenvalues of modulus larger than $t.$  The situation is a bit more complicated and will be discussed in Section 2.

For every positive integer $k,$ there are examples where $\nu$ has a density which is $k$ times differentiable. All Bernoulli convolutions with sufficiently large parameter $t<1$ outside an exceptional set of dimension zero, belong to this class \cite{hausdorffdim_shmerkin}. In this case, the first $k$ derivatives of the density are eigenfunctions of $H$ corresponding to the eigenvalues $t, t^2,...,t^k.$ Numerical experiments indicate that this holds if the density has only $k-1$ derivatives, and these values are the leading eigenvalues of the operator (Figure 1).  We shall give an explanation for this behavior.

\begin{figure}[h!]
\includegraphics[width=0.75\textwidth]{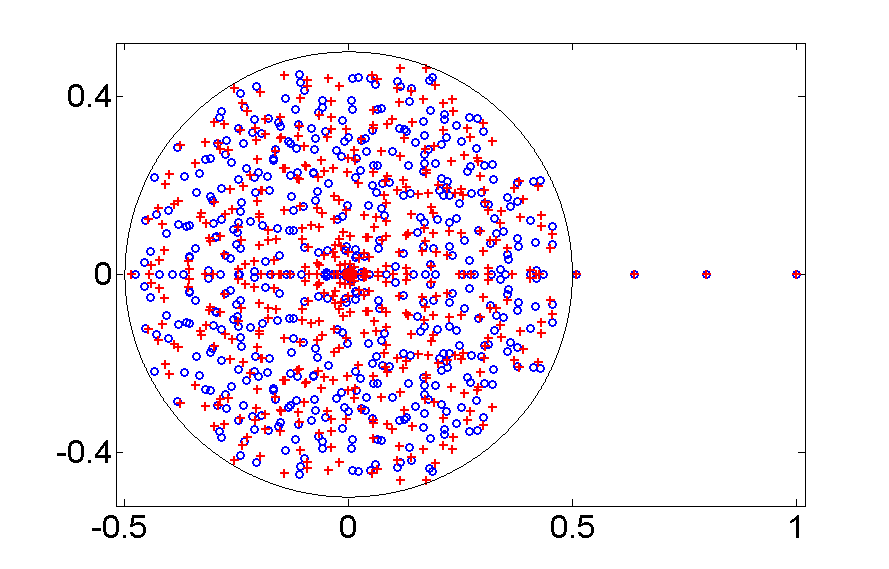}  
\caption{Eigenvalues of two matrix approximations of the Hutchinson operator for the Bernoulli convolution with $t=0.8.$  The matrix size, $N=499$ and 500, does not influence the leading eigenvalues $t^k, k=0,...,3$
while virtually all remaining eigenvalues are different. The circle has radius $\frac12 .$} \label{BC08}
\end{figure}

Hutchinson's theorem has led to well-known iteration algorithms which generate pictures of the self-similar measure $\nu .$ 
 These methods are fast and often give a nice impression of fractal structures.  In cases where $\nu$ has a density, however, they are not too accurate.  It is hard to decide from such approximations whether the function is smooth or monotone since one cannot distinguish genuine fractal structure and noise. Moreover, an approximation given by an iteration algorithm is just a set of data, not a mathematical  object.  We shall present an approximation by polynomials as an alternative.  See Figure 2.

The key to our study is the \emph{transfer operator} $T$ acting on continuous functions $h\in C(X)$ as follows.
\begin{equation}\label{trans} 
Th(x) = \sum_{i=1}^m p_i\cdot h(f_i(x)) \qquad\mbox{ for } x\in X\, .\end{equation}

The Hutchinson operator is the dual operator of $T$ and so their spectra are equal, see for instance \cite[Theorem 6.22]{kato}. Our first result says that $T$ has polynomial eigenfunctions which are easy to determine.
  
\begin{thm}\label{polyef} (Polynomial eigenfunctions of the transfer operator, real case)\\
The transfer operator $T$ has eigenvalues $\lambda_n = \sum_{i=1}^m p_i  t_i^n$ and corresponding polynomial eigenfunctions $q_n$ of degree $n$ for $n=0,1,2,...$ 
If the mappings $f_i$ have equal contraction factors $t_i=t,$
these eigenvalues are $\lambda_n = t^n$ for $n=0,1,2,...$.
\end{thm}

Different versions of Theorem \ref{polyef} will be discussed in Section 3. The proof provides a construction of the $q_n.$ Then we turn to the calculation of polynomial approximations of the self-similar measure $\nu .$ In the case that the support $X$ of $\nu$ is an interval, the following theorem holds with the usual notation 
\begin{equation}\label{scalarprod}
\langle f,g\rangle =\int_X f(x)g(x) dx\quad\mbox{ and }\quad\nu (f)=\int_X f(x) d\nu (x)\quad\mbox{ for }\ f,g\in C(X).
\end{equation}
 
\begin{thm}\label{polapprox} (Polynomial approximation of the self-similar measure)\\
Given the polynomial eigenfunctions $q_0=1,q_1,q_2,...$ of the transfer operator $T,$ let $v_n$ denote the polynomial on $X$ of degree $n$ which  satisfies the equations 

\begin{equation}\label{orth}
\langle v_n, q_0 \rangle =1\ ,\quad \langle v_n, q_k \rangle=0\quad\mbox{ for }\  k=1,\ldots,n\, . 
\end{equation}

Considered as density function, $v_n$ is the best approximation of the measure $\nu$ among polynomials of degree at most $n$ in the sense that 
\[  \frac{\nu (v_n)} {\| v_n\|_2} >  \frac{\nu (q)} {\| q\|_2}   \]
for all polynomials $q$ of degree $\le n$ which are not multiples of $v_n.$ \
If $\nu$ has an $L^2$ density $v$ then $v_n$ converges  to $v$ in the $L^2$ norm. \vspace{4ex}
\end{thm}

\begin{figure}[h]
\includegraphics[width=0.49\textwidth]{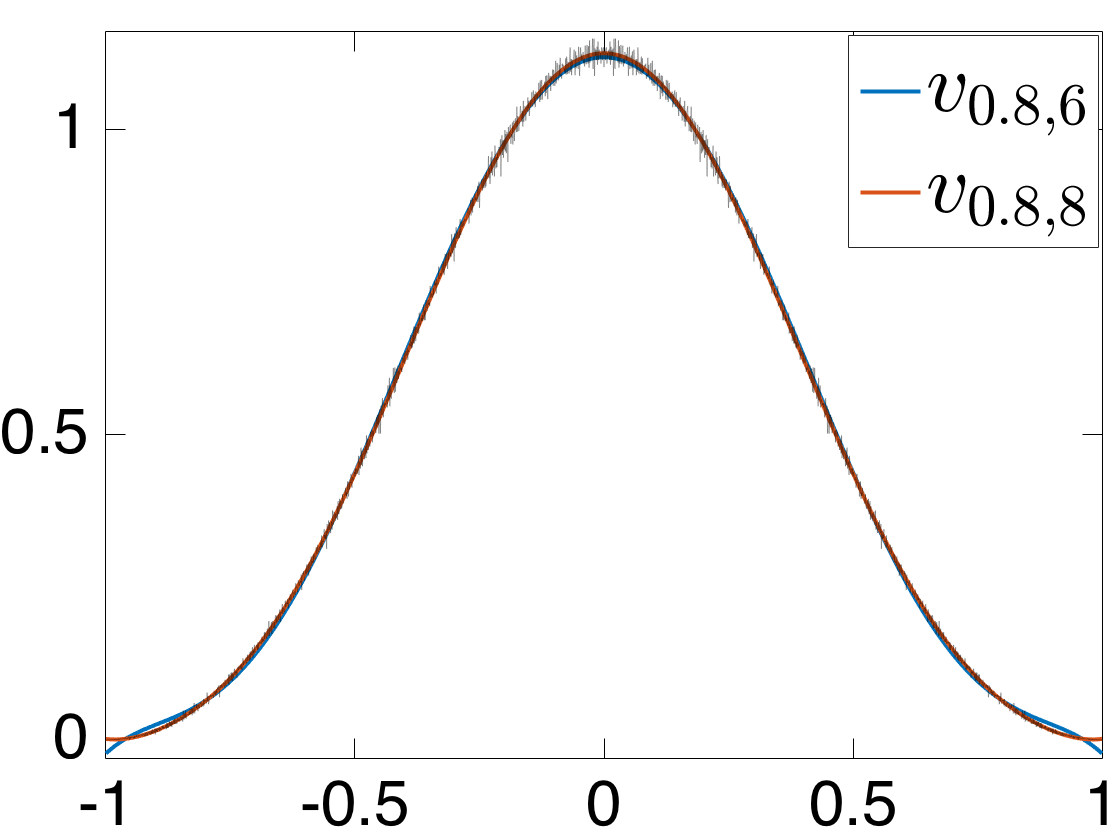}\hfill
\includegraphics[width=0.49\textwidth]{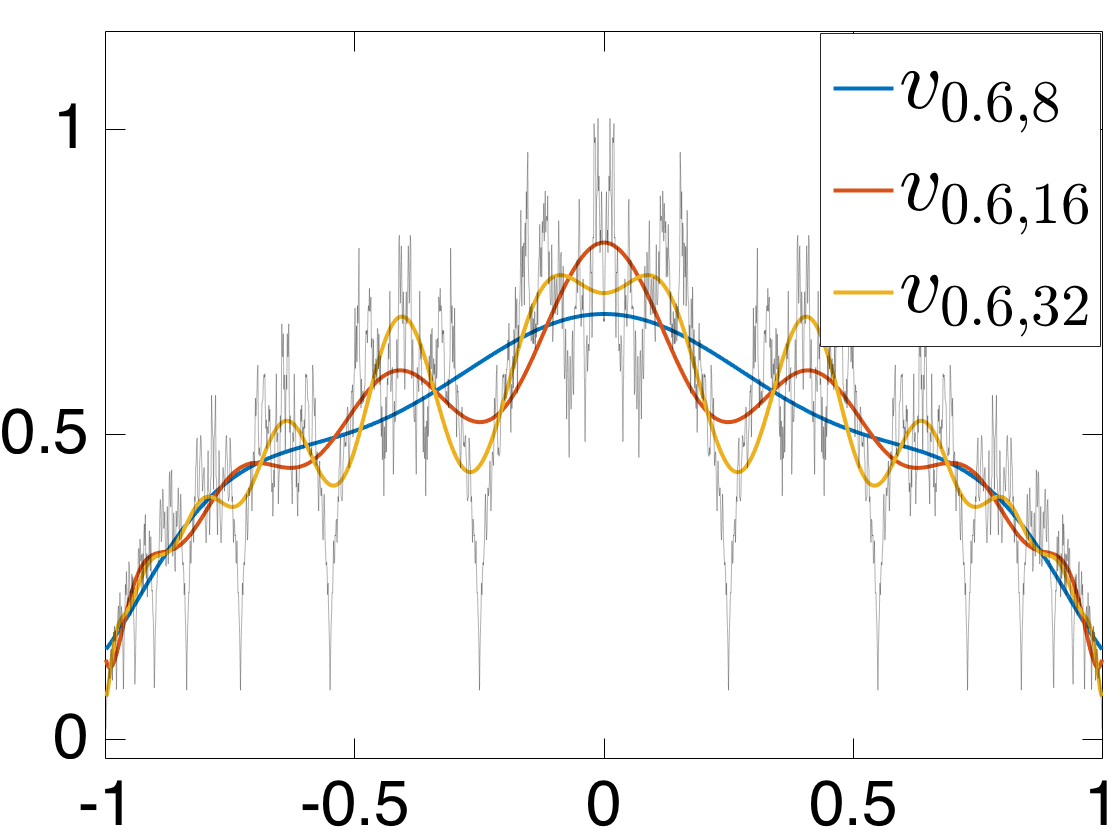}
\caption{Polynomial approximations $v_{t,n}$ of the Bernoulli convolution measure $\nu_t$ in the smooth case $t=0.8$ (left) and the more fractal case $t=0.6$ (right). In the left picture, one could not distinguish at this scale between the approximations of degree $n\geq 8$. The gray line shows a histogram with 2000 bins based on $2^{20}$ points generated by the 'chaos game' algorithm.}\vspace{4ex}
\label{polyax}
\end{figure}

Figure \ref{polyax} shows approximations of a smooth and a more fractal self-similar measure. In Section 4 we prove Theorem \ref{polapprox} and discuss modifications. In Section 5 we show that the problem of finding the approximating polynomials $v_n$ can be regarded as a moment problem, as considered in \cite{barnsleydemko,baileyrose}.
As a consequence, construction of the approximation $v_n$ amounts to solving a system of linear equations with a Hilbert matrix of coefficients. Section 6 deals with the case of Bernoulli convolutions, and
Theorem \ref{bcpols} will give an explicit analytical representation of the $q_n$ for this case. In the final section we indicate how singular measures can be studied with this approach.

\section{Experiments with the spectrum of Hutchinson's operator}
Before we go into proofs, let us discuss some experiments which have motivated this work. All Bernoulli convolutions will be considered on the interval $I=[-1,1].$  The parameter $t$ varies between 0.5 and 1, the figures refer to $t=0.8 .$  The mappings are
\begin{equation}\label{f12}
f_1(x)=tx-(1-t)\quad\mbox{ and }\quad f_2(x)=tx+(1-t)\, .\end{equation}
One way to think about self-similar measures is to consider $\nu$ as stationary distribution with respect to random iteration of the functions, popularized as 'chaos game'. We have a discrete time Markov process $(X_0,X_1,...)$ with kernel
\begin{equation}\label{ker}
p(x,B)=\pb(X_{n+1}\in B| X_n=x ) =\textstyle{\frac12}\,\delta_B(f_1(x)) + \textstyle{\frac12}\,\delta_B(f_2(x))
\end{equation}
To approximate by a Markov chain $(Z_0,Z_1,...)$ with $N$ states, we divide $I$ into $N$ disjoint intervals $I_k$ and introduce the transition probabilities   
\[ p_{kj} = \pb(Z_{n+1}\in I_j|Z_n \in I_k)=\frac 12\frac{|f_1^{-1}(I_j) \cap I_k|}{|I_k|} + \frac 12\frac{|f_2^{-1}(I_j) \cap I_k|}{|I_k|}
\]
where $|J|$ denotes the length of the interval $J.$
The matrix $T_N=( p_{kj})_{k,j=1,...,N}$ represents the transfer operator $T$ when multiplied by column vectors $(v_1,...,v_N)'$ which describe functions, and the Hutchinson operator $H$ when multiplied from the left by row vectors $(u_1,...,u_N)$ which describe measures. This matrix is always irreducible and aperiodic. (This follows from 
$p_{11}>0,\  p_{NN}>0,$ and from the fact that an interval $J$ is mapped either by  $f_1^{-1}$ or by  $f_2^{-1}$ to an interval which is $1/t$ times longer, unless $J$ contains the points $t$ and $1-t$ and thus will be mapped to state 1 and $N.$) 

\begin{figure}
\includegraphics[width=0.75\textwidth]{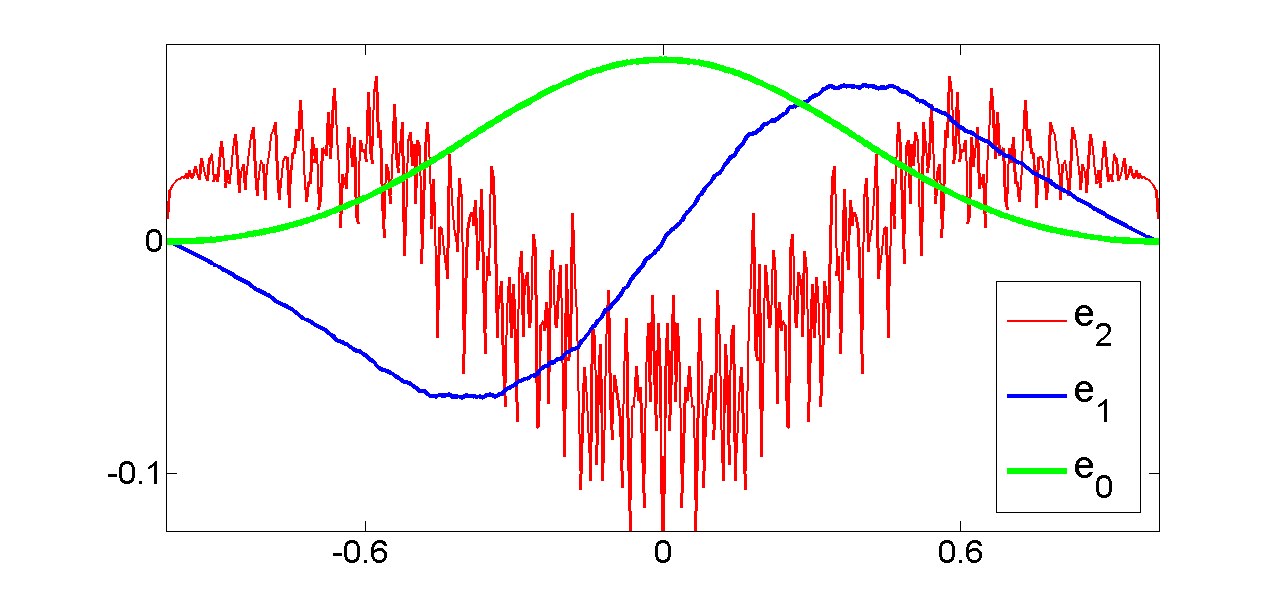}\\
\includegraphics[width=0.75\textwidth]{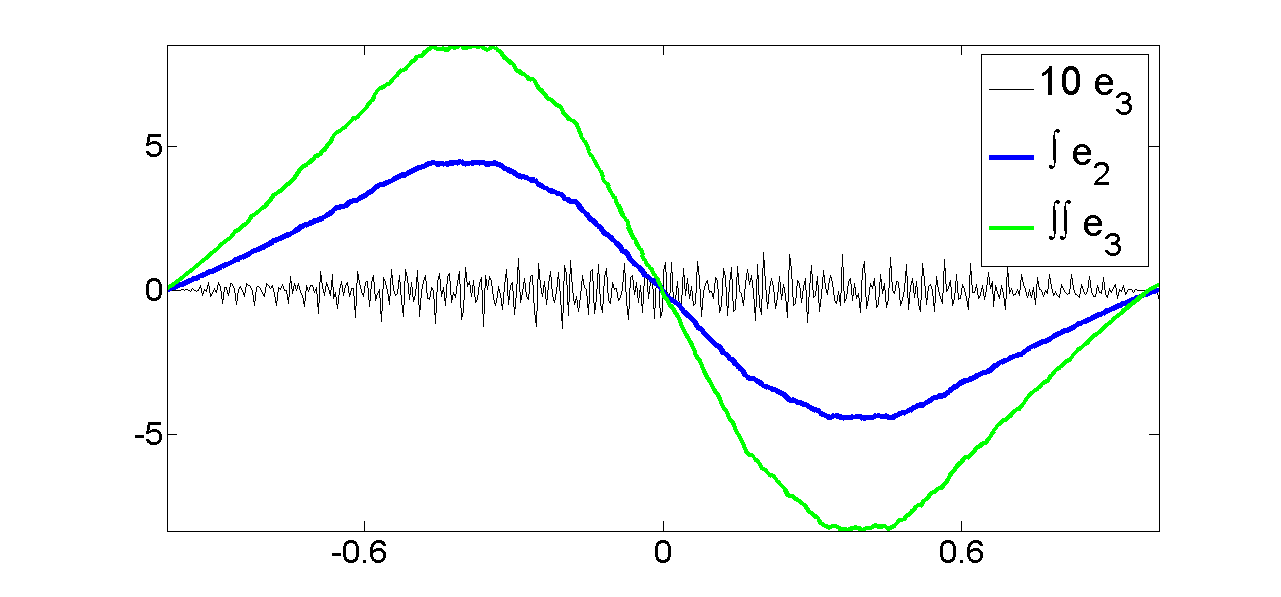}
\caption{Above: leading left eigenvectors $e_0,e_1,e_2$ of $T_N$ for the Bernoulli convolution with $t=0.8$ and $N=500.$  Below: scaled eigenvector  $e_3,$ integral of $e_2$ and iterated integral of $e_3.$ The latter two coincide with $e_1,$ up to a constant. Thus $e_1,e_2,e_3$ are the first 3 derivatives of the density $e_0$ of the self-similar measure.}
\label{BF08}
\end{figure}

So there is a unique left eigenvector of $T_N$ for the eigenvalue 1 which represents the self-similar measure, and a corresponding right eigenvector which is constant since we have a Markov matrix.  The convergence is fast: even if we take $N=10^6$ and start with uniform distribution, 60 iterations yield sufficient accuracy.  Calculations were performed with MATLAB. To determine eigenvalues and eigenvectors, we took $N$ between 400 and 2000. For each parameter $t,$ different $N$ were tested, with partitions into intervals of equal lengths and randomly perturbed partitions. 

The pattern of eigenvalues was always the same, as shown in Figure 1. The values $1,t,t^2,...,t^k$ were eigenvalues, as long as $t^k>\frac12 .$ The remaining eigenvalues were spread inside the circle $|\lambda|<\frac12 .$ In a few experiments, one or two percent of the points were slightly outside that circle. The distribution inside the circle could be more or less uniform, it could be more a ring, especially when $t$ is near $\frac12 ,$ or zero could be a multiple eigenvalue, but these properties usually changed with $N,$ as indicated in Figure 1. 

\begin{figure}[h]
\includegraphics[width=0.75\textwidth]{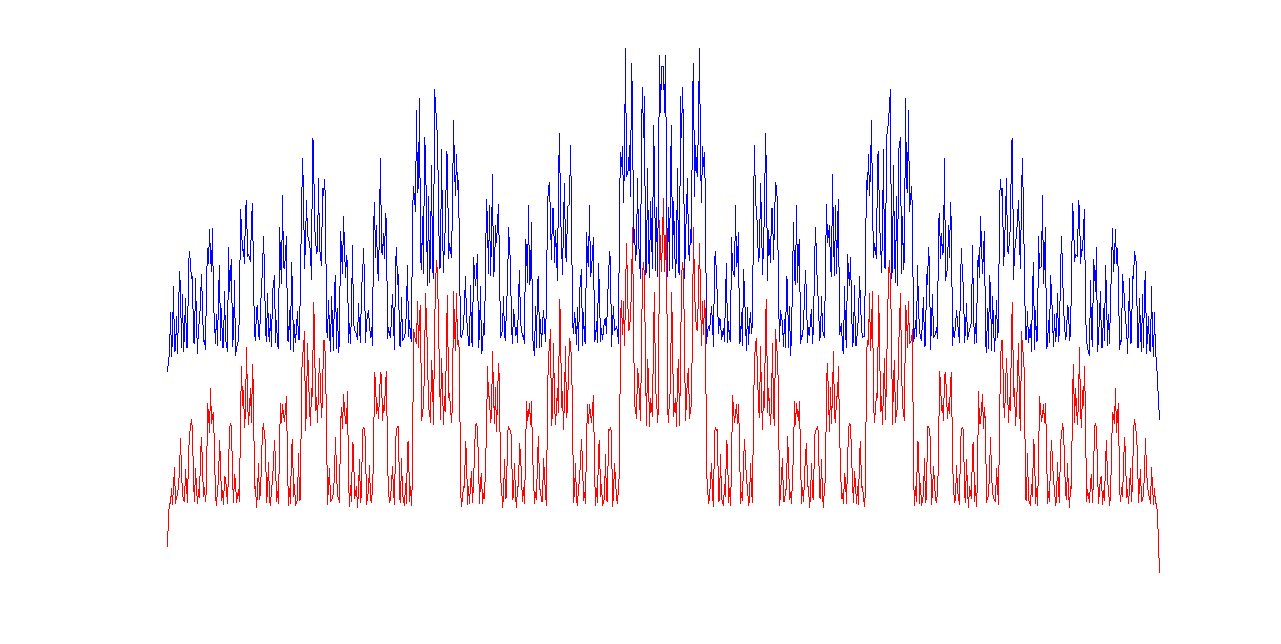}
\caption{$\beta=1/t=1.84$ was chosen near the Pisot parameter 1.8393... Below: the self-similar measure. Above: cumulative sum of second eigenvector. Although $\nu$ is not a differentiable function, the second eigenvector looks like a derivative of $\nu .$ }
\label{spec}
\end{figure}

The leading eigenvalues $t^k>\frac12, k=0,1,...$ did never change, and their left eigenvectors were also very stable. If a highly oscillating function is drawn for very different $N,$ there are natural differences, but for the smooth eigenfunctions deviations were very small. In Figure 3, the function which represents the measure $\nu$ looks like a normal distribution, and this is true whenever $t>0.8$ (cf. \cite{solomyak04}).

In Figure 3, the eigenvector $e_1$ represents the negative of the derivative of $e_0.$ The third eigenvector seems irregular, but when it is integrated (just taking cumulative sums) we get $e_1,$ multiplied by a constant. This indicates that we have the second derivative of $\nu .$ And even $e_3,$ two times integrated, gives the same result, up to a constant. Thus we have a third derivative of $\nu ,$ in spite of the fact that the second derivative seems a non-differentiable function. 

For the case that $\nu$ has a $k$ times differentiable density $\rho (x)$ it is easy to check that the derivatives are eigenvectors. The self-similarity condition \eqref{hut} can be reformulated as
\[  \rho(x)=\frac{\beta}{2}\left[ \rho(\beta x-1+\beta)+\rho(\beta x+1-\beta )\right]=: (H\rho)(x) \qquad\mbox{ with }\beta=\frac{1}{t}\, .\]
Taking derivatives we get  $\rho '=\beta\cdot H\rho ' ,$ that is, $H\rho '=t\rho ',$ and  $H\rho ''= t^2\cdot\rho ''$ etc. We cannot prove, however, that there are no other eigenfunctions with $|\lambda|>\frac12 .$  We also do not explain the fact that we get one more derivative than expected. As shown in Figure 4, this even holds when the density function is not at all differentiable. For the non-fractal case $t=\frac12$ this can be checked: $\nu$ has density equal $\frac12 ,$ and the distributional derivative $\frac12 ( \delta_{-1}-\delta_1 )$ is an eigenfunction of $H$ with eigenvalue $\frac12 .$ 

Pisot parameters play a large role in Bernoulli convolutions since Erd\"os proved 1939 that $\nu$ has no density in this case \cite{solomyak00,solomyak04}.  In our numerical study, they did not look very special, although the pattern of eigenvalues seemed somewhat different, and for the golden mean a discontinuous eigenfunction with $\lambda=-\frac12$ was identified. Figure 4, for example, is very near to the so-called Tribonacci parameter, and still we have the 'derivative' as second eigenvector.

All eigenvalues with $|\lambda|<\frac12$ have extremely noisy left eigenvectors, although they are symmetric, odd or even functions. Since they change with $N$ and with slight changes of partition, they can be interpreted as noise. This is highly correlated noise, however. If these functions are integrated two times, they still have a fractal appearance. If white noise is integrated two times, a smooth function will result.

The right eigenvectors of $T_N$ look much more regular. Those which correspond to $\lambda=t^k>\frac12$ are polynomials, as shown in Figure 5. This will be proved below.  The other right eigenvectors represent noise. They are all odd or even, and rather smooth functions -- more smooth than left eigenvectors after two integrations. The duality between polynomials and special functions somewhat resembles the classical Sturm-Liouville operators which lead to special functions and associated polynomials of Legendre, Hermite etc, although our operator $T$ is not self-adjoint and we have only a finite set of 'special functions'.

\begin{figure}
\includegraphics[width=0.75\textwidth]{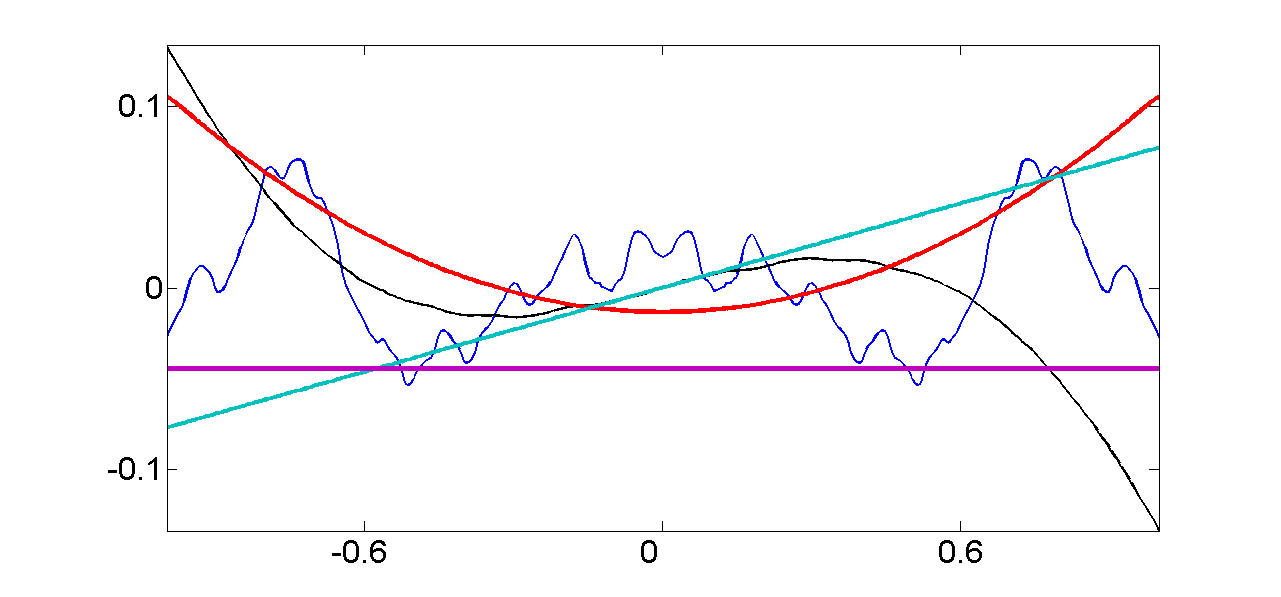}
\caption{Right eigenvectors are polynomials for $\lambda>\frac12 .$ First eigenvector with $|\lambda|<\frac12$ shown for comparison. Parameters as in Figure 3.}
\label{BP08}
\end{figure}

What about the circle $|\lambda |=\frac12$ ?  For $t=\frac12$ there is an explanation. In this case $Th(x)=\frac12 (h(\frac{x-1}{2})+ h(\frac{x+1}{2}))$ and $T$ has all points $\lambda$ of the unit circle as eigenvalues. An eigenfunction for $\lambda$ is $h(x)=\sum_{n=0}^\infty \lambda^n \exp(\pi k 2^n{\bf i}x)$ where $k$ is a fixed odd number. See Driebe \cite{driebe} and the literature quoted there.  $h(x)$ is a Weierstrass function. It is well-known that $h(x)$ is differentiable only for $|\lambda|<\frac12 ,$  see \cite{baranski1} and its references.  Apparently, numerical experiments only detect eigenvalues of differentiable functions. This point seems worth a further study.

\section{Eigenpolynomials of the transfer operator}

We now prove Theorem \ref{polyef}, which concerns the existence of polynomial eigenfunctions of the transfer operator for affine IFS on $\reals$. We consider affine maps $f_i(x)=t_ix+v_i$ on $\reals$ with $0<t_i<1$ for $i=1,...,m.$
The self-similar set $X$ given by the $f_i$ is the unique compact set which fulfills the equation $X=\bigcup_{i=1}^m f_i(X).$
Moreover, we have probabilities $p_i>0$ with $\sum_{i=1}^m p_i=1,$ and the self-similar measure $\nu$ on $X$ is defined by equation \eqref{hut}. The only assumption which we need is that $X$ consists of more than one point. Then $X$ has infinitely many points, and each polynomial on $\reals$ is determined by its values on $X.$ 

\begin{proof}[Proof of Theorem \ref{polyef}] 
The transfer operator $T$ defined in \eqref{trans} with mappings $f_i(x)=t_i x + v_i$ for $i=1,\ldots,m$ maps the polynomial $b_k(x)=x^k$ to 
\begin{equation}\label{Tbk}
Tb_k(x)= \sum_{i=1}^m p_i (t_ix+v_i)^k = \sum_{l=0}^k\sum_{i=1}^m p_i \binom kl t_i^l v_i^{k-l}x^l
\end{equation}
which is a linear combination of $b_0(x),b_1(x),\ldots,b_k(x)$ and is thus a polynomial of degree $k$. This shows that the space of polynomials of degree $\leq k$ is an invariant subspace of $T$. 
The action of $T$ on the polynomials of degree $\leq n$ is then represented, with respect to the basis polynomials $b_0,b_1,b_2,\ldots,b_n$, by the matrix

$$ T^{(n)} =  \begin{pmatrix}
1 & \ast & \ast  &  \ldots   & \ast \\[1ex]
0	& \sum_{i=1}^m p_i t_i & \ast    & \ldots &  \ast   \\[1ex]
0 & 0 &  \sum_{i=1}^m p_i t_i^2 & \ldots & \ast \\[1ex]
\vdots & &\ddots &\ddots & \vdots \\[1ex]
0 & 0 & \ldots  & 0 & \sum_{i=1}^m p_i  t_i^n
\end{pmatrix} $$
with entries 

\begin{equation}\label{Tn}
T^{(n)}_{lk} = \binom kl \sum_{i=1}^m p_i t_i^l v_i^{k-l}\quad \text{for} \;\; 0\leq l \leq k  \leq n
\end{equation}
and 0 else. Note that $Tb_k=\sum_{l=0}^k T_{lk}^{(n)}b_l$. The invariant subspaces show up in the upper-triangular shape. 
The eigenvalues of  $T^{(n)}$ are the diagonal entries $\lambda_k = \sum_{i=1}^m p_i t_i^k,$ $k=0,1,\ldots,n$. Since $0<t_i<1$, they satisfy $\lambda_0=1 > \lambda_1 > \ldots > \lambda_n$. All eigenvalues are different and the right eigenvectors are polynomial eigenfunctions $q_0=1, q_1,\ldots,q_n$ of $T$ of degree $0,1,\ldots,n$ respectively.
\end{proof}

Theorem \ref{polyef} can be extended to some other IFS. 

\begin{rem}\emph{(Complex case)}\label{remcplx}\\
Consider an IFS on $\cplx$ with mappings $f_i(z)=t_iz+v_i$ with $t_i,v_i\in \cplx$, $i=1,\ldots,m$, and probabilities $p_1,\ldots,p_m$. In this case the self-similar set $X$ is a subset of $\cplx$. The operator $T$ defined in equation \eqref{trans} acts on the complex vector space $C(X)$ of continuous functions on $X$ and has eigenfunctions and eigenvalues as in the real case. But now we have complex polynomials, and we will add their conjugates. The proof of Theorem 1 goes through when we regard the $b_k(z)=z^k$ as complex polynomials. With the additional assumption that the complex eigenvalues $\lambda_0=1,\lambda_1,\lambda_2,\ldots$ are all different, we see that $T$ has complex eigenpolynomials $q_0=1,q_1,q_2,\ldots$. The 'conjugate polynomials' $1,\overline{z},\overline{z}^2,\ldots$ yield other invariant subspaces of $T$. Namely, $\overline{z}^k$ is mapped into a linear combination of $1,\overline{z},\ldots,\overline{z}^k$. The action of $T$ on $1,\overline{z},\ldots,\overline{z}^n$ is then represented by the matrix $\overline{T^{(n)}}$ which yields the conjugate eigenpolynomials $\overline{q_1},\overline{q_2},\ldots$ for the eigenvalues $\overline{\lambda_1},\overline{\lambda_2},\ldots$.
\end{rem}

\begin{rem}\emph{($d$-dimensional case)}\\ 
On $\reals^d$ the polynomials $q(x)=q(x_1,...,x_d)$ of degree $n$ are linear combinations of the monomials $b_k(x)=x_1^{k_1}x_2^{k_2}... x_d^{k_d}$ where $k=(k_1,...,k_d)$ runs through the vectors of nonnegative integers with $k_1+...+k_d=n.$ We consider an affine IFS $f_i(x)=L_ix+v_i, i=1,...,m$ where each $L_i$ is a $d\times d$ matrix. Then it is easy to check that the space of polynomials of degree $\le n$ is an invariant subspace of the transfer operator $T$ for each $n.$

To get a triangular matrix $T^{(n)},$ one has to order the basis vectors $b_k$ as above with respect to ascending degree $n=0,1,2,...$ Moreover, one has to impose conditions on the maps. We assume diagonal matrices, as used for Bedford-McMullen carpets and Gatzouras-Lalley carpets: $L_i$ has diagonal vector $\ell_i=(\ell_{i1},...,\ell_{id})$ and zeros outside the diagonal.  
Then $b_k(L_ix)=b_k(\ell_i)b_k(x),$ and equation \eqref{Tbk} obtains the form 
\begin{equation*}\label{Tbkd}
Tb_k(x)= \sum_{i=1}^m p_i b_k(L_ix+v_i) = \sum_{i=1}^m p_i b_k(\ell_i) b_k(x) + \mbox{ terms of degree smaller than }n\, . 
\end{equation*}

So we get an upper triangular matrix $T^{(n)}$ with diagonal entries $\lambda_k=\sum_{i=1}^m p_i b_k(\ell_i)$ which are the eigenvalues. If all $f_i$ have the same matrix $L$ with diagonal $\ell ,$ then $\lambda_k=b_k(\ell) .$
In case that the $\lambda_k$ are all different, the eigenpolynomials $q_k$ form a basis for the space of polynomials of degree $\leq n$. A similar remark holds if the $L_i$ are upper triangular matrices. 
\end{rem}

\section{Approximation of the self-similar measure}

We proved that the $(n+1)$-dimensional vector space $P_n$ of polynomials of degree at most $n$ is invariant under $T$ and contains $n$ linearly independent non-constant eigenpolynomials. These eigenpolynomials are all orthognal to the self-similar measure $\nu$ in the sense that
\begin{equation}\label{ornu} \nu (q_k)=\int_X q_k(x) d\nu (x)=0 \quad\mbox{ for } k=1,2,... \end{equation}
This follows from the duality of the operators $T$ and $H,$ and the fact that the eigenvalue of $q_k$ fulfils  $\lambda_k\not= 1$ for $k\ge 1$ by a standard argument:  
\[ \nu (q_k)=(H\nu )(q_k)=\nu (Tq_k) =\nu (\lambda_kq_k)=\lambda_k \nu (q_k)\ .\]
So there is one dimension left in the space, and Theorem \ref{polapprox} says that the vector $v_n$ which is orthogonal to all eigenpolynomials, considered as a density function, is the best possible approximation of $\nu $ on $P_n.$  Moreover, the space of all polynomials is dense in $C(X),$ so that the sequence of density functions $v_n$ is in some respect the best approximation of $\nu$ by continuous density functions. We now prove Theorem \ref{polapprox}, assuming that the self-similar set $X$ is an interval.  We use the scalar product \eqref{scalarprod} and the norm $\| q\|=\sqrt{\langle q,q\rangle}.$ 

\begin{proof}[Proof of Theorem \ref{polapprox}]
Let $\rm H$ be the $n$-dimensional hyperplane generated by the eigenpolynomials $q_1,...,q_n$ in the $(n+1)$-dimensional space $P_n$ of polynomials of degree $\le n .$ For $h\in \rm H$ we have $\nu(h)=0$ and $\langle v_n,h\rangle =0.$ 
A polynomial $q\in P_n\setminus \rm H$ can be written as $q=\alpha v_n +h$ with $\alpha\not= 0$ and $h\in\rm H.$ This implies
\[ \nu(q)=\alpha\nu(v_n)\qquad\mbox{ and }\quad \langle v_n,q\rangle  =\alpha \langle v_n,v_n\rangle . \]
We see that $\nu(q)/\langle v_n,q\rangle =\nu(v_n)/\langle v_n,v_n\rangle $ is a constant. The assumption $\nu(1)=\langle v_n,1\rangle=1$ shows that this constant equals one, and that $\nu(v_n)>0.$  The equality  $\nu(q)=\langle v_n,q\rangle$ says that $v_n$ is the \emph{generating vector} of the linear form $\nu$ on the Hilbert space $P_n.$ The Cauchy-Schwartz inequality $\nu(q)\le \| q\|\cdot \| v_n\|$ directly implies the inequality of Theorem \ref{polapprox}.

Now we show $L^2$ convergence of the $v_n,$ for the case that $\nu$ has an $L^2$ density $v.$ In this case the above proof implies that $v_n$ is the orthogonal projection from $v$ onto the space $P_n.$ Thus $\| v_n-v\|$ equals the distance from $v$ to $P_n$ which converges to zero by the Weierstrass theorem and the fact that continuous functions are dense in $L^2$, see for example \cite[Theorem~1.10.18, Proposition 1.10.4]{tao}.
\end{proof}

The case that $\nu$ does not possess an $L^2$ density will be discussed in Section 7.
$L^2$ convergence applies in particular to the Bernoulli convolutions treated in Section 6 since their self-similar measures belong to $L^2$ for almost every parameter, cf. \cite{solomyak00,solomyak04}.
We assumed that $X$ is an interval but Theorem \ref{polapprox} also holds under different assumptions.

\begin{rem}\emph{(IFS with open set condition)} \\
If $f_1,...,f_m$ satisfies the open set condition, it is known that the $s$-dimensional Hausdorff measure $\mathcal{H}^s$ on $X$ with $\sum_{i=1}^m t_i^s=1$ is positive and finite \cite{hutchinson,falconer}. We can replace the Lebesgue measure in the above proof by $\mathcal{H}^s,$ in particular in the definition \eqref{scalarprod} of $\langle f,g\rangle .$ Using the Stone-Weierstrass theorem for compact Hausdorff spaces $X$ (see e.g. \cite[Theorem~1.10.18]{tao}), we obtain
Theorem \ref{polapprox}.
\end{rem} 

\begin{rem}\emph{($d$-dimensional case)} \\  Based on Remark 2, Theorem \ref{polapprox} can be extended to the multidimensional case which is work in progress. For complex maps, as mentioned in Remark 1, there is the problem that the algebra generated by complex polynomials and their conjugates, needed for the Stone-Weierstrass theorem, is much larger than the generated vector space. 
\end{rem}

\section{Moments and construction of the best approximation}

The $j$-th moment of the probability measure $\nu$ is the integral  $m_j=\nu(b_j)=\int_X x^j d\nu ,$ that is, the value of $\nu$ on the $j$-th basis vector in the polynomials, for $j=0,1,2,...$ Characterization of measures by moments is a classical topic \cite{shohat}, for self-similar measures see \cite{barnsleydemko,baileyrose}. 
The eigenpolynomials of $T$ provide a recursive method to compute moments of $\nu$. 

\begin{pro}\label{mompol}(Eigenpolynomials and moments of the self-similar measure) \\
Let $q_n(x)= a_{n,0} + a_{n,1} x + \ldots + a_{n,n-1}x^{n-1} + x^n$ be the eigenpolynomials of the transfer operator with leading coefficient one.  The moments $m_n = \int_X x^n d\nu, \; n=1,2,\ldots,$ can be computed by 
$$ m_n = - a_{n,0}  - a_{n,1} m_1 - \ldots - a_{n,n-1} m_{n-1} $$
\end{pro}

\emph{Proof. }
From  $\int_X q_n d\nu = 0$ for $n=1,2,\ldots$ we conclude 
\[ a_{n,0} \int_X d\nu + a_{n,1}\int_X x d\nu + \ldots  + \int_X x^n d\nu = a_{n,0} + a_{n,1} m_1+ \ldots +  m_n = 0\, .\qquad\hfill\qed\vspace{1ex}\]

Now we have the moments, and we show how to compute the approximating polynomials $v_n$ by a system of linear equations with a well-known coefficient matrix.

\begin{pro}\label{coeffs}(System of linear equations for the approximating polynomials) \\
The orthogonality relation \eqref{orth} which defines the polynomial $v_n$ is equivalent to the condition 
\begin{equation}\label{momen} \langle v_n,b_k \rangle = m_k \quad\mbox{ for }\ k=0,1,\ldots,n\, .\end{equation}
Consequently, the coefficient vector $u=(u_0,u_1,\ldots,u_n)'$ of $v_n(x)=u_0+u_1x+\ldots+u_n x^n$ satisfies the linear system of equations
$$ Gu = m$$
where $m= (m_0,m_1,\ldots,m_n)'$ and $G$ is the Hilbert matrix with entries
$$ G_{ij} = \int_X x^{i+j} dx, \quad \;\; i,j=0,1,\ldots,n .$$
\end{pro}

\begin{proof}
The first assertion can be derived recursively from Proposition \ref{mompol}, or from the following argument. Because of \eqref{ornu}, condition \eqref{orth} says that $v_n$ and $\nu$ generate identical linear forms on the space of polynomials generated by $q_0=1,q_1,...,q_n.$ Equation \eqref{momen} expresses the same fact, referring to the standard base $b_0,...,b_n.$ So the two conditions are equivalent.

The explicit form of the equations \eqref{momen} is easily determined:\\
$$\langle v_n,b_j\rangle  = \int_X (u_0+u_1x+\ldots+u_n x^n) x^j \; dx = u_0 \int_X x^j dx + u_1 \int_X x^{j+1} dx + \ldots + u_n \int_X x^{n+j} dx $$
This is the $j$-th term of $Gu .$ 
\end{proof}

Of course, given the $q_k,$ condition \eqref{orth} can directly be considered as a system of equations. The above reformulation establishes a connection with classical matrices \cite{hilbert} which have been thoroughly studied by many authors, mainly for their notoriously bad condition number, see e.g. \cite{choi}, \cite[Example 3.3]{beckermann}.  Moreover, the system $Gu=m$ can be solved without determining the eigenfunctions $q_k$ when moments are directly calculated in a recursive way.

\section{Bernoulli convolutions}

The IFS for the Bernoulli convolutions with parameter $t\in (0,1)$ consists of the mappings $f_1(x)=tx+v_1$ and $f_2(x)=tx+v_2$ with $v_1\neq v_2$ and probabilities $p_1=p_2=\tfrac 12$. Since there is only one conjugacy class for each $t$ we can choose $v_1$ and $v_2$ arbitrarily. We work with the mappings
\begin{equation}\label{fi}
f_1(x)=tx-1+t, \quad f_2(x)=tx+1-t.
\end{equation}

We denote the corresponding self-similar measure by $\nu_t$. In case $t\in (0,\tfrac 12)$ the measure $\nu_t$ is supported on a Cantor set and is necessarily singular. For all $t\in [\tfrac 12, 1)$ the support is the same interval $X =[-1,1],$ which was a reason for choosing the mappings \eqref{fi}. 
In this case a density for $\nu_t$ might exist, and will exist for almost all $t>\frac12 .$ The question for which $t$ exactly we have a density has been studied by many authors over more than 75 years. See the surveys of Peres, Schlag, and Solomyak \cite{solomyak00}, Solomyak \cite{solomyak04}, and the recent work of Shmerkin \cite{hausdorffdim_shmerkin}. 

We will now prove a formula for the eigenpolynomials of the transfer operator for the Bernoulli convolutions. Then we  compute the moments of $\nu_{t}$ with Proposition \ref{mompol}.  Finally, we will use these moments to compute approximations of $\nu_t$ with Proposition \ref{coeffs}. 

\begin{thm}(Explicit formula for Bernoulli convolutions)\label{bcpols}\\
For the Bernoulli convolution with mappings $f_1(x)=tx-1+t$ and $f_2(x)=tx+1-t$ the polynomial eigenfunctions $q_n$ of the transfer operator correspond to the eigenvalues $\lambda_n = t^n$ and can be written as 
\begin{equation}\label{polybc} 
q_n(x) = \sum_{k=0}^n a_{n,k} x^k\quad \; \; n=0,1,2,...  \end{equation}
where $a_{n,n} = 1$, $ a_{n,n-1}=0,$ and the other coefficients are determined recursively by 
\[ 
a_{n,k} = \frac 1{t^{n-k}-1}  \sum_{s=1}^{\floor{\frac {n-k}2}} \binom{k+2s}{k} (1-t)^{2s} a_{n,k+2s} \quad \; \; k=0,1,\ldots,n-2.
\]
Thus $q_n$ is an even function for even $n$ and an odd function for odd $n$.
\end{thm}

\begin{proof}
The action of the transfer operator on polynomials of degree $n$ is represented by the matrix $T^{(n)}$ defined in the proof of Theorem \ref{polyef}. Substituting the Bernoulli convolution parameters $t_1=t_2=t$, $p_1=p_2=\tfrac 12$ and $v_1=-1+t$, $v_2=1-t$ in  \eqref{Tn} yields the entries $T^{(n)}_{lk} = \binom kl (\tfrac 12 t^l (-1+t)^{k-l} + \tfrac 12 t^l(1-t)^{k-l})$ for $l\leq k \leq n$ and 0 else. Non-zero entries appear only if $k-l$ is non-negative and even, in which case $T^{(n)}_{lk} = \binom kl t^l (1-t)^{k-l}$:

$$ T^{(n)} =  \left( \begin{smallmatrix}
1  & \quad \quad 0 \quad \quad & \binom 20 (1-t)^2  & 0 & \binom 40 (1-t)^4 & 0   &\quad \quad \ldots\\[1ex]
0	&  \quad\quad t\quad \quad & 0   & \binom 31 t(1-t)^2 & 0  & \binom 51 t (1-t)^4 & \quad \quad \ldots  \\[1ex]
 & \quad\quad\;0 \quad\quad \;&  t^2 & 0 & \binom 42 t^2(1-t)^2 & 0 & \quad \quad \ldots \\[1ex]
\vdots  & \quad\quad \quad\quad & 0 & t^3 & 0 & \binom 53 t^3 (1-t)^2 & \quad \quad \ldots \\[2.2ex]
    & \quad\quad \quad \quad &   &0 & t^4 & 0 &  \quad \quad \ldots\\[1ex]
& \quad\quad \quad\quad & & & \ddots & \ddots & \quad \quad \quad \\[1ex]
& \quad\quad \quad\quad & & &  &  & \quad \quad \quad \\[1ex]
0 & \quad \quad\quad \quad &  & & & 0 &  \quad \quad t^n
\end{smallmatrix} \right)
.$$
The diagonal entries $\lambda_0=1, \lambda_1=t, \ldots, \lambda_n=t^n$ are eigenvalues of $T^{(n)}$ and thus of $T$ too. Let $q_n(x)=a_{n,0} +a_{n,1} x+ a_{n,2} x^2 + \ldots + a_{n,n-1} x^{n-1} + a_{n,n} x^n$ with $a_{n,n}=1$ be the  eigenpolynomial corresponding to the eigenvalue $\lambda_n=t^n$. The eigenvalue equation $Tp_n=t^n p_n$ can be rewritten as
\begin{equation}\label{Teq}
T^{(n)} a = t^n a 
\end{equation}
in terms of the coefficient vector $a=(a_{n,0}, a_{n,1}, a_{n,2}, \ldots, a_{n,n-1}, 1)'$.
Suppose $a_{n,k+1},\ldots,a_{n,n}$ are known for some $k \in \lbrace 0,1,\ldots, n-1 \rbrace$. Equating the $k-$th coordinate in (\ref{Teq}) yields 
$$ a_{n,k} t^k +  \sum_{s=1}^{\floor{\tfrac {n-k}2}}  \binom{k+2s}{k} t^k(1-t)^{2s}  a_{n,k+2s}=t^n a_{n,k}\, .$$
So we have the recursion $a_{n,k} = \frac 1{t^{n-k}-1}  \sum_{s=1}^{\floor{\frac {n-k}2}} \binom{k+2s}{k} (1-t)^{2s} a_{n,k+2s} .$  Setting $k=n-1$ shows that $a_{n,n-1}=0$. Thus $q_n$ is a polynomial of degree $n$ with either odd or even powers of $t$.
\end{proof}

With the aid of the  eigenpolynomials we compute the moments of $\nu_t$, as shown in Proposition \ref{mompol}.
Then we turn to the computation of the approximations
$$v_{t,n}(x) = u_0 + u_1x + \ldots + u_n x^n$$
defined in Theorem \ref{polapprox}. Let $m=(m_0,m_1,\ldots,m_n)'$ be the first $n+1$ moments of $\nu_{t}$. By Proposition \ref{coeffs}, the coefficients $u=(u_0,u_1,\ldots,u_n)'$ satisfy the equation
\begin{equation}\label{lgs}
Gu = m
\end{equation}
with matrix $G_{ij}=\int_{-1}^1 x^{i+j} dx = \tfrac 1{i+j+1} (1-(-1)^{i+j+1})$ for $i,j=0,1,\ldots,n$. This Hilbert matrix has the shape

$$ G = 2 \begin{pmatrix}
1	&  0	& \frac 13	 & 0  & \frac 15 & \ldots     \\
0	& \frac 13	& 0 & \frac 15  & 0  & \ldots \\
\frac 13	& 0 &  \frac 15 & 0  & \
\frac 17 & \ldots \\
0 &  \frac 15 & 0 & \frac 17  & 0 & \ldots \\
\vdots &\cdot&\cdot&\cdot & \cdot & \ldots
\end{pmatrix} $$

Equation \eqref{lgs} splits into an even part for computing $u_0,u_2,\ldots ,$ and an odd part for computing $u_1,u_3,\ldots$.  The  equation $G_* u_* = m_*$ for the odd coefficients $u_*=(u_1,u_3,\ldots)'$ has right hand side $m_*=(m_1,m_3,...)'=0$  and  $u_*$ vanishes since we approximate a measure which is symmetric around the origin.
The coefficients $ u^*=(u_0,u_2,\ldots, )'$ of even powers of $x$ depend on the even moments $m^*=(m_0,m_2,\ldots)'$ and are determined by 
$ G^* m^* = u^*$
with matrix 
$$ G^* = 2 \begin{pmatrix}
1& \frac 13	  & \frac 15 & \frac 17 & \ldots     \\
 \frac 13 & \frac 15   & \frac 17 & \frac 19 & \ldots \\
\frac 15	 &  \frac 17   & \frac 19  & \frac 1{11} &\ldots \\
  \frac 17  &  \frac 19  & \frac 1{11} & \frac 1{13} & \ldots\\
\vdots &\cdot  & \cdot & \cdot & \ldots 
\end{pmatrix}.$$

$G^*$ is, up to the factor 2, the Hilbert matrix studied in \cite{hilbert} and shown to be invertible. Note that the entries do not depend on $t.$ This is good for numerical computations, since one has to store the inverse only once for all $t.$

Figure \ref{polyax} shows approximations of Bernoulli convolutions for $t=0.6$ and $t=0.8$. \vspace{3ex}

\section{Remarks on singular measures}
What happens if $\nu$ is a singular measure?  We tried to prove that our approximating polynomials $v_n,$ taken as measures $\nu_n(B)=\int_B v_n(x)\, dx\, ,$ will converge weakly to $\nu ,$ in the sense that $\nu_n(f)\to \nu(f)$ for continuous functions $f.$ Fortunately, a referee found a gap in our proof, and an error in an attempt to repair the proof. We are very grateful for the referee's careful work!  Let us briefly discuss the approximation of singular measures.

\begin{pro}\label{sing}(Properties of approximating polynomials)\\
Under the assumptions of Theorem  \ref{polapprox}, the $v_n$ satisfy
\begin{enumerate}
\item $\langle v_n, q\rangle = \nu (q)$  for all polynomials $q$ of degree $\le n\, .$ 
\item $\langle v_n, q\rangle \ge 0$ for all polynomials $q$ of degree $\le n$ with $q(x)\ge 0$ for $x\in X\, .$
\item $\| v_n\|^2\le \max_{x\in X} v_n(x)\ .$
\item If $\nu$ does not admit an $L^2$ density then $\| v_n\|\to\infty .$ 
\end{enumerate} 
\end{pro}

\begin{proof} (1) was proved in Section 4, and (2) is an immediate consequence. For (3), let $M= \max_{x\in X} v_n(x)\, .$ Then $q=M-v_n$ is a positive function. (2) implies $\langle v_n, M-v_n\rangle\ge 0$ since $\langle 1,v_n\rangle =1.$  

For (4), we note that  $\| v_n\|$ is an increasing sequence since $v_n$ is the orthogonal projection of $v_{n+1}$ onto $P_n.$ If the $\| v_n\|$ are bounded, they form a weakly relatively compact set in $L^2,$ and a subsequence of $v_n$  weakly converges to some $v\in L^2.$ Since (1) implies $\langle v, q\rangle = \nu (q)$ for all polynomials $q$ which are dense in $L^2$ as well as in $C(X),$ this $v$ must be an $L^2$ density of $\nu .$ So for singular $\nu$ the sequence $\| v_n\|$ tends to infinity. This kind of argument is known as uniform boundedness principle, cf. \cite[Corollary~1.7.7]{tao}. 
\end{proof}

\begin{figure}[h!]
\includegraphics[width=0.97\textwidth]{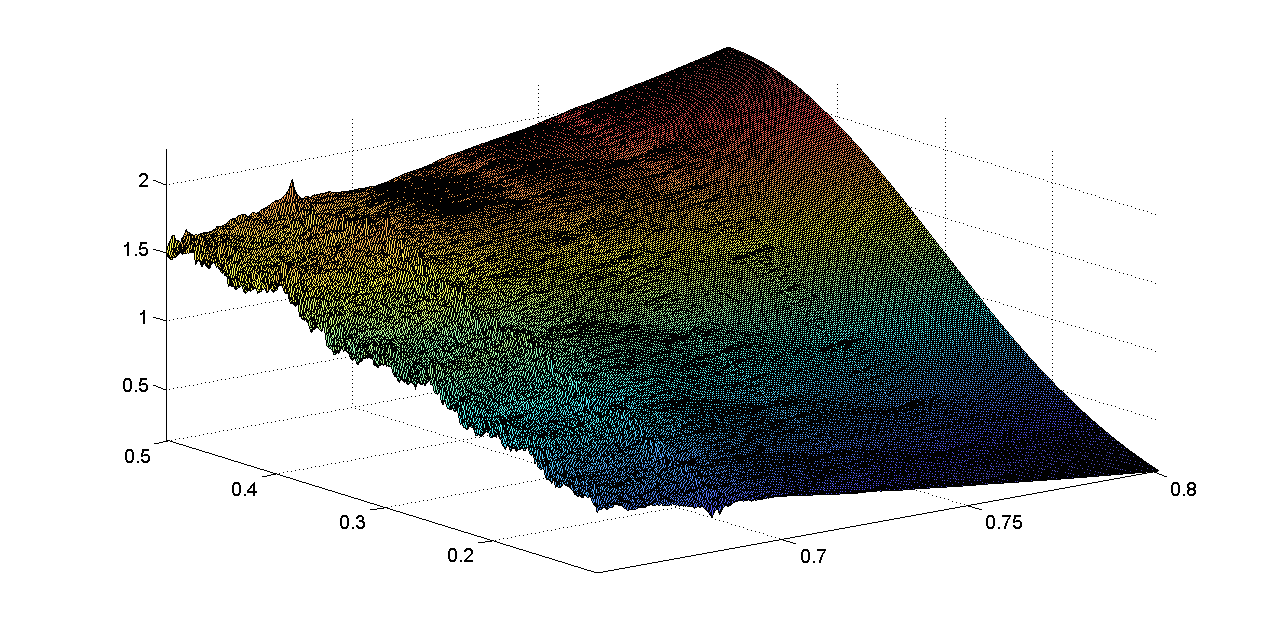}  
\caption{Polynomial approximation to the two-dimensional Bernoulli density for $0.65\le t\le 0.8\, , \, 0.1\le x\le 0.5 .$ Apparent singularities at $t=0.68$ correspond to the Pisot number $\beta=1.4656 .$ For the first two Pisot numbers  the singularity is less obvious. } \label{BCfamily}
\end{figure}

The proposition indicates that approximation of singular measures requires unbounded and highly oscillating polynomials.  For instance, (3) shows that the set $S=\{ x\, |\, v_n(x)\ge \frac{M}{2}\}$ fulfils $M\ge\| v_n\|^2\ge\int_S v_n^2\ge (\frac{M}{2})^2\cdot |S|$ and thus has Lebesgue measure $|S|\le\frac4M$ which is very small for very large $M.$ Such polynomials seem not to be helpful for a study of specific singular measures. \medskip

It seems appropriate, however, to study singular measures and polynomials within the context of a parametric family where most  measures have an $L^2$ density. Due to their construction, the $v_n$ continuously depend on the parameter for each $n.$ We have good polynomial approximations in $L^2$ for all regular parameter values. For the singular values, we can study the singularities of the two-variable function which is obtained in this way. Figure \ref{BCfamily} shows a two-variable polynomial approximating the family of Bernoulli convolutions for   $0.65\le t\le 0.8\, .$
There are three Pisot numbers $\beta= 1.3247,\, 1.3803, \, 1.4656$ where $\nu$ is known to be singular \cite{solomyak00}. The figure shows a very smooth function: a polynomial of degree 28 in $t$ and $x.$  The singularities at $\beta=1.4656$ are obvious, those at the other parameters are barely visible.  A study of Bernoulli convolutions as a two-parameter $L^2$ function was initiated in \cite{bandt15} and will be continued elsewhere.

\bibliographystyle{plain}
\bibliography{lt_ep}
\vspace{3ex}

\noindent
Christoph Bandt, Helena Pe\~na\\
Institute of Mathematics, University of Greifswald, Germany\\
\url{bandt@uni-greifswald.de, thaki1110@gmail.com}

\end{document}